\documentclass[a4paper,10pt]{amsart}
\usepackage[english]{babel}
\usepackage{amsmath,amssymb,amsthm}
\usepackage[pdftex]{color}
\usepackage[bookmarks=true,hyperindex,pdftex,colorlinks, citecolor=MidnightBlue,linkcolor=MidnightBlue, urlcolor=MidnightBlue]{hyperref}
\usepackage[dvipsnames]{xcolor}
\usepackage{enumerate}
\usepackage{lineno}
\usepackage[edges]{forest}
\usepackage{mathtools}
\usepackage{graphicx}
\usepackage{subfig}

\usetikzlibrary{decorations.markings}
\tikzset{neg/.style={
		decoration={markings,
			mark= at position 0.5 with {
				\node[transform shape] (tempnode) {$\setminus$};
			}
		},
		postaction={decorate}
}}

\newtheorem{theorem}{Theorem}[section]
\newtheorem{proposition}[theorem]{Proposition}
\newtheorem{corollary}[theorem]{Corollary}
\newtheorem{lemma}[theorem]{Lemma}

\theoremstyle{definition}

\newtheorem*{definition*}{Definition}
\newtheorem*{proposition*}{Proposition}
\newtheorem*{theorem*}{Theorem}
\newtheorem*{corollary*}{Corollary}
\newtheorem*{example*}{Example}
\newtheorem*{problem*}{Problem}

\theoremstyle{remark}

\newcommand{\degree}{\operatorname{deg}}

\newcommand*{\probb}[2]{\mathbb{P}_{#1}\left\{ #2 \right\}}
\newcommand*{\tmix}{t_{\operatorname{mix}}}
\newcommand*{\trel}{t_{\operatorname{rel}}}
\newcommand*{\Var}{\operatorname{Var}}

\begin{document}
	
	\title{No cutoff in Spherically symmetric trees}
	
	\author[Chiclana]{Rafael Chiclana}
	
	\author[Peres]{Yuval Peres}
	
	\address{Kent State University, Kent, Ohio}
	\email{rchiclan@kent.edu, yuval@yuvalperes.com}


	\date{May 25th, 2022}

	\begin{abstract}
		We show that for lazy simple random walks on finite spherically symmetric trees, the ratio of the mixing time and the relaxation time is bounded by a universal constant. Consequently, lazy simple random walks on any sequence of finite spherically symmetric trees do not exhibit pre-cutoff; this conclusion also holds for continuous-time simple random walks.  This answers a question recently proposed by Gantert, Nestoridi, and Schmid. We also show that for lazy simple random walks on finite spherically symmetric trees, hitting times of vertices are (uniformly) non concentrated. Finally, we study the stability of our results under rough isometries.
	\end{abstract}
	
	\maketitle
	
	\thispagestyle{plain}

\section{Introduction}
Random walks on certain families of graphs exhibit the cutoff phenomenon, which is a fast transition in the convergence to the stationary distribution (see, e.g.,  \cite{ben-hamou2019comparing} and \cite{lubetzky2010cutoff}). In this note we focus on  families of trees. Peres and Sousi presented in \cite{peres2015total} a first example of a sequence of trees on which the lazy simple random walk exhibits cutoff. More recently, Gantert, Nestoridi, and Schmid gave a sufficient condition to guarantee that the lazy simple random walk on a sequence of trees exhibits cutoff (see \cite[Theorem 1.6]{gantert2020cutoff}). They also showed that, in some sense,  cutoff on trees is a rare phenomenon. More concretely, the authors presented in \cite{gantert2020cutoff} some estimations on the mixing time and relaxation time to show that the families of the (continuous-time) simple random walks on several classes of trees, including Galton-Watson trees, do not exhibit cutoff. Among other results, it is proved that if $T$ is an infinite spherically symmetric tree of maximum degree $\Delta$, and $(T_n)_{n \in \mathbb{N}}$ is a family of trees obtained by truncating $T$ to its first $n$ levels, for all $n \in \mathbb{N}$, then the family of the (continuous-time) simple random walks on $(T_n)_{n\in \mathbb{N}}$ does not exhibit cutoff. The main goal of this paper is to answer Question 6.1 in \cite{gantert2020cutoff}, that asks whether the assumption in the above result on having a bounded maximum degree can be relaxed. As a consequence of our main result, if $(T_n)_{n \in \mathbb{N}}$ is any sequence of finite spherically symmetric trees, then the family of the (continuous-time) simple random walks on $(T_n)_{n \in \mathbb{N}}$ does not exhibit cutoff. This answers the last question, but also shows that the trees $T_n$ do not need to be truncations of a single infinite spherically symmetric tree. 

Theorem 1.1 in \cite{chen2013comparison} shows that cutoff for the (continuous-time) simple random walks on $(T_n)_{n \in \mathbb{N}}$ is equivalent to cutoff for the (discrete-time) lazy simple random walks on $(T_n)_{n \in \mathbb{N}}$, so it is enough to study the last ones. The next result shows that the ratio of the mixing time and the relaxation time of the lazy simple random walk on a spherically symmetric tree is bounded by a universal constant. Then criterion (\ref{criterion}) gives that $(T_n)_{n\in\mathbb{N}}$ does not exhibit pre-cutoff, which is a weak version of cutoff (see Section \ref{section2}).
\begin{theorem}\label{theo}
	There exists a universal constant $C>0$ so that the lazy simple random walk on a spherically symmetric tree $T$ satisfies
	\begin{equation}\label{eqtheo} \trel \geq C \tmix.
	\end{equation}
	Consequently, if $(T_n)_{n\in\mathbb{N}}$ is a sequence of finite spherically symmetric trees, then the family of the lazy simple random walks on $(T_n)_{n \in \mathbb{N}}$ does not exhibit pre-cutoff.
\end{theorem}

Although the proof of Theorem \ref{theo} does not optimize the constant, it proves that we can take $C= \frac{1}{144}$.

Let $(X_t)$ be the lazy simple random walk on a connected graph $G=(V,E)$. Given $A \subseteq V$, write $\tau_A$ for the first time that $(X_t)$ hits $A$, that is, $\tau_A=\inf\{t\geq 0 \colon X_t\in A\}$. When $A=\{a\}$ we simply write $\tau_a$. Similarly, $\tau_a^+=\inf\{t\geq 1\colon X_t=a\}$. In Section \ref{section4}, we study concentration of hitting times of vertices on spherically symmetric trees. This is also studied in \cite{norris2017surprise} for general Markov chains. For the lazy simple random walk on a simple graph, one can use Chebyshev's inequality to extract from Theorem 1.2 in \cite{norris2017surprise} the following lower bound of the variance.

\begin{proposition}
	There is a universal constant $D>0$ so that the lazy simple random walk on a simple graph $G=(V,E)$ with $n\geq 2$ vertices satisfies
	\[ \Var_x(\tau_y)\geq D\frac{\mathbb{E}_x(\tau_y)^2}{(\log{n})^2}, \quad \forall \, x, y \in V.\]
\end{proposition}

The above bound can be improved if the simple graph is a spherically symmetric tree.

\begin{theorem}\label{theo:hitting}
	There is a universal constant $C'>0$ so that the lazy simple random walk on a spherically symmetric tree $T$ satisfies
	\[ \Var_x(\tau_y^+)\geq C' \mathbb{E}_x(\tau_y^+)^2 \quad\forall \, x,y \in T.\]
\end{theorem}
In particular, this gives nonconcentration for return times when $x=y$. Some results for general graphs are studied in \cite{gurel2013nonconcentration}. Although the proof of Theorem \ref{theo:hitting} does not focus on optimizing the constant, it shows that we can take $C'=\frac{1}{484}$. This is not true for general graphs, or even general trees, as the family of trees constructed in \cite{peres2015total} shows. This also follows from Lemma 2.3 in \cite{norris2017surprise}, where the authors consider a slight modification of the previous family that, for some distinct vertices $x$ and $y$, satisfies
\[\Var_x(\tau_y)=O\left (\frac{\mathbb{E}_x(\tau_y)^2}{\log n}\right ).\]
Finally, in Section \ref{sectionrough} we study the stability of Theorem \ref{theo} under rough isometries. In general, cutoff is not preserved by rough isometries (see Theorem 2 in \cite{hermon2018on}). However, as Proposition \ref{prop} shows, it is preserved when we consider trees. This observation is also made in \cite[Remark 1.7]{hermon2018on}. Thus, the following result is an immediate consequence of Theorem \ref{theo} and Proposition \ref{prop}.

\begin{corollary}\label{cor:rough}
	Let $(T_n)_{n\in\mathbb{N}}$ be a family of spherically symmetric trees with bounded degree $\Delta$. For $n \in \mathbb{N}$, let $T'_n$ be a tree roughly isometric to $T_n$ with constants $\alpha$ and $\beta$ not depending on $n$. Then, the lazy simple random walk on $(T'_n)_{n\in\mathbb{N}}$ does not exhibit cutoff.
\end{corollary}

\section{Preliminaries}\label{section2}
Given two probability measures $\mu$, $\nu$ on a set $V$, their \textit{total variation distance} is
\[\|\mu-\nu\|_{TV} = \max_{A\subseteq V}|\mu(A)-\nu(A)|.\]
Let $(X_t)$ be the \textit{lazy simple random walk} on a connected graph $G=(V,E)$, that is, every step with probability $\frac{1}{2}$ the chain either stays at the same vertex or goes to an adjacent vertex chosen uniformly at random. Given $a$, $b \in V$, we write $\tau_{a,b}=\inf\{t\geq \tau_b\colon X_t=a\}$, where $X_0=a$. The \textit{commute time} is $t_{a\leftrightarrow b}=\mathbb{E}_a(\tau_{a,b})$. The transition matrix of $(X_t)$ is denoted by $P$, and its stationary distribution is denoted by $\pi$. It is well known that $P$ is \textit{reversible}, that is, $\pi(x)P(x,y)=\pi(y)P(y,x)$ for all $x$, $y \in V$. The $\varepsilon$\textit{-mixing time} of $(X_t)$ is
\[ \tmix(\varepsilon)=\inf \left \{ t \geq 0 \colon \max_{x \in V} \|P^t(x,\cdot)-\pi\|_{TV} \leq \varepsilon\right \} \quad \forall \, \varepsilon \in (0,1).\]
The \textit{mixing time} of $(X_t)$ is $\tmix=\tmix(\frac{1}{4})$. It is well known that all eigenvalues of the transition matrix of a reversible lazy chain are positive. Let $\lambda_2$ be the second greatest eigenvalue of $P$. The \textit{spectral gap} of the chain is defined by $\gamma= 1-\lambda_2$. The \textit{relaxation time} is defined by $\trel=\frac{1}{\gamma}$. The following characterization of the spectral gap (see \cite[Remark 13.8]{levin2017markov}) will be useful to prove our main result.
\begin{equation}\label{gamma}
	\gamma= \min_{\substack{f \in \mathbb{R}^V \\ \operatorname{Var}_\pi(f)\neq 0}} \frac{\mathcal{E}(f)}{\operatorname{Var}_\pi(f)},
\end{equation}
where $\mathcal{E}(f)\coloneqq \frac{1}{2}\sum_{x,y \in V} |f(x)-f(y)|^2 \pi(x) P(x,y) $ is the \textit{Dirichlet form} of $f$.

Let $(G_n)_{n \in \mathbb{N}}$ be a sequence of graphs and let $(\tmix^n(\varepsilon))_{n \in \mathbb{N}}$ be the collection of $\varepsilon$-mixing times of random walks on $(G_n)_{n\in \mathbb{N}}$. We say that this family of random walks on $(G_n)_{n\in \mathbb{N}}$ exhibits \textit{cutoff} if for any $\varepsilon \in (0,1)$
\[ \lim_{n\to\infty} \frac{\tmix^n(\varepsilon)}{\tmix^n(1-\varepsilon)}=1.\]
The cutoff phenomenon was first verified in \cite{diaconis1981generating}, and was formally introduced in the seminal paper of Aldous and Diaconis \cite{aldous1986shuffling}. Ever since then, the cutoff phenomenon has been widely studied for many specific examples of Markov chains. As a weaker condition, the family of random walks on $(G_n)_{n \in \mathbb{N}}$ is said to exhibit \textit{pre-cutoff} if
\[ \sup_{0<\varepsilon<\frac{1}{2}} \limsup_{n\to\infty} \frac{\tmix^n(\varepsilon)}{\tmix^n(1-\varepsilon)} < \infty.\]
A necessary condition to have pre-cutoff is that for some $\varepsilon \in (0,1)$ (or equivalently, for all $\varepsilon \in (0,1)$)
\begin{equation}\label{criterion}
	\lim_{n\to \infty} \frac{\tmix^n(\varepsilon)}{\trel^n} = \infty,
\end{equation}
where $(\trel^n)_{n\in \mathbb{N}}$ denotes the collection of relaxation times of the family of random walks on $(G_n)_{n \in \mathbb{N}}$ (see \cite[Proposition 18.4]{levin2017markov}). Despite Aldous' finding that in general Condition \ref{criterion} is not sufficient to have a cutoff (see Chapter 18 of \cite{levin2017markov}), it is believed to be sufficient for many families of Markov chains, such as lazy simple random walks on trees (see \cite{basu2017characterization}).

Recall that a \textit{tree} is a connected graph with no cycles. A \textit{rooted tree} has a distinguished vertex $o$, called the \textit{root}. The \textit{depth} of a vertex $v$ is its graph distance to the root. The \textit{height} of a tree is the maximum depth. A \textit{level} of the tree consists of all vertices at the same depth. A \textit{leaf} is a vertex of degree one and a \textit{branching point} is a vertex of degree at least $3$. A rooted tree $T$ is \textit{spherically symmetric} if all vertices at the same level have the same degree. We write $\degree_k$ for the degree of the vertices at level $k$.

\section{No cutoff in Spherically symmetric trees}

In this section, we answer Question 6.1 in \cite{gantert2020cutoff} by showing that the family of the (continuous-time) simple random walks on a sequence of finite spherically symmetric trees $(T_n)_{n \in \mathbb{N}}$ does not exhibit cutoff. Recall that \cite[Theorem 1.1]{chen2013comparison} allows us to restrict our study to the (discrete-time) lazy simple random walk. In view of (\ref{criterion}), the desired result follows from the bound on the ratio of the mixing time and the relaxation time that Theorem \ref{theo} provides.

Recall that there exists a universal constant $C_1>0$ for which, for any vertex $y$ of a tree $T$, the mixing time for the simple random walk on $T$ is bounded as follows:
\[\tmix \leq C_1 \max_{x \in V} \mathbb{E}_x (\tau_{y}).\]
See \cite[Lemma 9.3]{peres2015mixing}, where it is proved for central nodes, and \cite[Proposition 3.1]{gantert2020cutoff} for a reference of the general result. As the following lemma shows, when the tree is spherically symmetric, for a specific choice of the vertex $y$ we can take $C_1=12$.

\begin{lemma}\label{lemmacoupling}
	Let $T$ be a finite spherically symmetric tree of height $h$ and let $v$ be a vertex at level $h$. If $\deg_0\geq 2$ or $T$ has no branching points, let $v^*$ be the root of $T$. Otherwise, let $v^*$ be the closest branching point to the root. Then, the lazy simple random walk on $T$ satisfies
	\begin{equation}\label{boundcoupling}
		\tmix \leq 4(\mathbb{E}_o(\tau_{v^*})+2\mathbb{E}_v(\tau_{v^*})).
	\end{equation}
\end{lemma}
\begin{proof}
	Consider the following coupling $(X_t,Y_t)$ of two lazy simple random walks, started from states $x$ and $y$ on the tree. At each move, toss a coin to decide which of the two chains moves. The chosen chain will move to one neighbor chosen uniformly at random, while the other one stays at the same position. Run these two chains according to this rule until they are at the same level of the tree. After that, the chain $(X_t)$ will evolve as the lazy simple random walk, and the chain $(Y_t)$ will move closer to or further to the root if and only if $(X_t)$ moves closer to or further to the root. Once they are at the same vertex, $(Y_t)$ mimics $(X_t)$. Let $\tau_{\operatorname{couple}}= \inf \{t \geq 0 \colon X_s=Y_s \mbox{ for all } s\geq t\}$. Then, Corollary 5.5 in \cite{levin2017markov} gives
	\[ \tmix\leq 4 \max_{x,y\in V} \mathbb{E}(\tau_{\operatorname{couple}}).\]
	Finally, observe that no matter what the initial states $x$ and $y$ are, the expected time until both chains are at level $h$ is bounded by the expected time that the lazy simple random walk needs to go from $o$ to level $h$, which is bounded by $\mathbb{E}_o(\tau_{v^*})+\mathbb{E}_v(\tau_{v^*})$. Then, by the time the chains go back to the vertex $v^*$ they must be equal, so we have 
	\[\max_{x,y\in V} \mathbb{E}(\tau_{\operatorname{couple}})\leq\mathbb{E}_o(\tau_{v^*})+2\mathbb{E}_v(\tau_{v^*}).\qedhere\]
\end{proof}

The next lemma gives a lower bound for the relaxation time of reversible Markov chains in terms of hitting times of sets when the chain starts from stationary. It follows from Lemma 10 in \cite{aldous1992inequalities}, but we give a direct proof for completeness.

\begin{lemma}\label{aldous} 
	Let $(X_t)$ be a reversible Markov chain on a state space $V$ with stationary distribution $\pi$. For any subset $A\subseteq V$ with $0<\pi(A)<1$,
	\[ \trel \geq \frac{\pi(A)}{\pi(A^\mathsf{c})}\mathbb{E}_\pi(\tau_A).\]
\end{lemma}

\begin{proof}
	Define $f\colon V \longrightarrow \mathbb{R}$ by $f(x)=\mathbb{E}_x(\tau_A)$. By conditioning of the first step, for any $x \notin A$ we have
	\[ f(x)=1+\sum_{y \in V} P(x,y)f(y)= 1+Pf(x).\]
	A well-known fact, that can be proved with a simple computation, is that  the Dirichlet form of $f$ satisfies $\mathcal{E}(f)=\langle (I-P)f,f\rangle_\pi$, where $I$ and $P$ denote the identity and transition matrix, respectively. Since $f$ vanishes on $A$, we have that
	\begin{equation}\label{eq:dirichlet}
		\mathcal{E}(f)=\langle f-Pf,f\rangle_\pi= \langle 1,f\rangle_\pi=\sum_{x \notin A} f(x)\pi(x)= \mathbb{E}_\pi(f)=\mathbb{E}_\pi(\tau_A).
	\end{equation}
	Write $\mu$ for $\pi$ conditioned on $A^\mathsf{c}$ and observe that
	\[ \mathbb{E}_\pi(f^2)=\pi(A^\mathsf{c}) \mathbb{E}_\mu(f^2)\geq\pi(A^\mathsf{c}) \mathbb{E}_\mu(f)^2=\frac{1}{\pi(A^\mathsf{c})}\mathbb{E}_\pi(f)^2. \]
	Thus, 
	\[\Var_\pi(f)= \mathbb{E}_\pi(f^2)-\mathbb{E}_\pi(f)^2\geq\left(\frac{1}{\pi(A^\mathsf{c})}-1\right )\mathbb{E}_\pi(f)^2= \frac{\pi(A)}{\pi(A^\mathsf{c})}\mathbb{E}_\pi(f)^2=\frac{\pi(A)}{\pi(A^\mathsf{c})}\mathbb{E}_\pi(\tau_A)^2.\]
	The result now follows from the characterization of the spectral gap (\ref{gamma}).
\end{proof}

The following simple lemma will be useful for future estimations.

\begin{lemma}\label{lemmaconcave}
	Given $h\in \mathbb{N}$, let $f\colon\{0,\ldots,h+1\}\longrightarrow \mathbb{R}$ be a concave increasing function satisfying $f(0)=0$, and let $w$ be a probability measure on $\{0,\ldots,h+1\}$ such that $w(0)\leq \ldots \leq w(h)$ and $w(0)\leq w(h+1)$. Then,
	\[ \mathbb{E}_w(f^2) \geq \frac{1}{7}f(h+1)^2.\]
\end{lemma}

\begin{proof}
	First, if $h=1$ we can use the concavity of $f$ and $w(0)\leq w(2)$ to get
	\[\mathbb{E}_w(f^2)= w(1)f(1)^2 + w(2)f(2)^2\geq \left (\frac{w(1)}{4}+w(2)\right )f(2)^2 \geq \frac{f(2)^2}{4}.\]
	For $h\geq 2$, write $w_1$ and $w_2$ for $w$ conditioned on $\{0,h+1\}$ and $\{1,\ldots,h\}$, respectively. Observe that $\mathbb{E}_w(f^2)\geq\min\{ \mathbb{E}_{w_1}(f^2), \mathbb{E}_{w_2}(f^2)\}$. On the one hand, we have
	\[ \mathbb{E}_{w_1}(f^2) = \frac{w(h+1)f(h+1)^2}{w(0)+w(h+1)} \geq \frac{f(h+1)^2}{2}.\]
	On the other hand, define $\tilde{f}\colon\{1,\ldots,h\}\longrightarrow \mathbb{R}$ by $\tilde{f}(j)=\frac{j}{h}f(h)$ for every $j\in \{1,\ldots,h\}$. Using that $f$ is concave we get $f\geq \tilde{f}$. Moreover, since $\tilde{f}$ is increasing and $w_2(1)\leq w_2(2)\leq \ldots \leq w_2(h)$ we have
	\begin{equation}\label{eqconcav} \mathbb{E}_{w_2}(f^2)\geq \mathbb{E}_{w_2}(\tilde{f}^2)\geq \frac{1}{h}\sum_{j=1}^h \tilde{f}(j)^2 =\frac{f(h)^2}{h^3} \sum_{j=1}^h j^2\geq \frac{f(h)^2}{3}.
	\end{equation}
	Since $f$ is concave we also have $f(h)\geq \frac{h}{h+1}f(h+1)\geq \frac{2}{3}f(h+1)$, so the result follows from (\ref{eqconcav}).\qedhere
\end{proof}

\begin{proof}[Proof of Theorem \ref{theo}]
	Let $T=(V,E)$ be a finite spherically symmetric tree of height $h$. If $\degree_0\geq 2$, then set $\ell=0$. Otherwise, let $\ell$ be the level at which we find the closest branching point to the root. If $\degree_0=1$ and there are no branching points (and so the graph is a segment), set $\ell=h$.
	
	Let $S=\{o=v_0,v_1,\ldots,v_\ell\}$ be the set of vertices that belong to the (possible) initial segment of the graph. Notice that if $\degree_0\geq 2$, then $S=\{o\}$. In the case when the graph is not a segment, let $T(1), \ldots, T(r)$ be the connected components that we obtain after removing the vertices of $S$. Let $A$ be the union of the first $\lfloor\frac{r}{2}\rfloor$ connected components and let $B$ be the union of the last $\lfloor\frac{r}{2}\rfloor$ connected components. If $r$ is odd, set $C=T(\lfloor\frac{r}{2}\rfloor+1)$. Otherwise, set $C=\emptyset$ (see Figure \ref{figure1}).
	
	\begin{figure}[t]
		\centering
		\scalebox{0.45}{
			\begin{forest}
				for tree={math content,
					grow=south,
					circle, draw, minimum size=2ex, inner sep=4.5pt,
					s sep=10mm, l sep=3mm
				}
				[o,fill={white}, inner sep=6.5,font=\huge[v_1,fill={white},inner sep=5,font=\LARGE[\vdots,align=center,inner sep=4,font=\Large[v_\ell,fill={white},inner sep=6,font=\LARGE[,fill={white} [,fill={white}[,fill={white}[,fill={white}[,fill={white}[,fill={white}[,fill={white}]]][,fill={white}[,fill={white}[,fill={white}]]]]]] ][,fill={white} [,fill={white}[,fill={white}[,fill={white}[,fill={white}[,fill={white}[,fill={white}]]][,fill={white}[,fill={white}[,fill={white}]]]]]] ]
				[,fill={gray} [,fill={gray}[,fill={gray}[,fill={gray}[,fill={gray}[,fill={gray}[,fill={gray}]]][,fill={gray}[,fill={gray}[,fill={gray}]]]]]] ] [,fill={black} [,fill={black}[,fill={black}[,fill={black}[,fill={black}[,fill={black}[,fill={black}]]][,fill={black}[,fill={black}[,fill={black}]]]]]] ][,fill={black} [,fill={black}[,fill={black}[,fill={black}[,fill={black}[,fill={black}[,fill={black}]]][,fill={black}[,fill={black}[,fill={black}]]]]]] ]]]]]]]
		\end{forest}}
		\caption{Example of a spherically symmetric tree. Labeled vertices correspond to the set $S$. White, black, and gray vertices correspond to the sets $A$, $B$, and $C$, respectively.}
		\label{figure1}
		
	\end{figure}
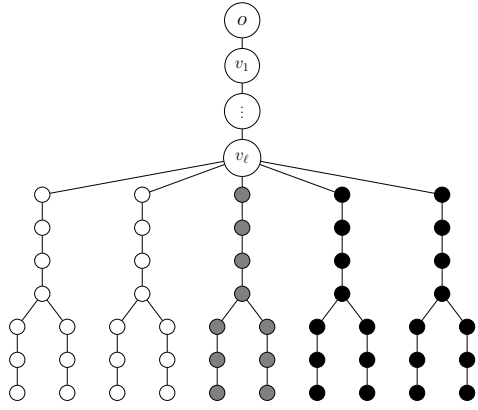
	
	In view of Lemma \ref{lemmacoupling}, we distinguish two cases. Assume that $\frac{5}{2}\mathbb{E}_{o} (\tau_{v_\ell})\leq\mathbb{E}_{v} (\tau_{v_\ell})$, where $v$ is a leaf of $T$ at level $h$ (and so the graph is not a segment). Take $D=A\cup C \cup S$. Then, Lemma \ref{aldous} gives 
	\begin{equation}\label{eq:case1aldous}
		\trel \geq \frac{\pi(D)}{\pi(D^\mathsf{c})}\mathbb{E}_\pi(\tau_D)= \frac{\pi(D)}{\pi(B)}\mathbb{E}_\pi(\tau_D).
	\end{equation}
	For every $j=1,\ldots,h$, let $V_j$ be the set of vertices at level $j$. Define $g_1\colon \{\ell,\ldots,h\}\longrightarrow \mathbb{R}$ by $g_1(j)=\mathbb{E}_{v_j}(\tau_D)$, where $v_j \in V_j$. Observe that
	\[ \mathbb{E}_\pi(\tau_D)=\sum_{x\in B} \mathbb{E}_x(\tau_D) \pi(x)= \sum_{j=\ell+1}^{h} g_1(j)\pi(V_j\cap B)= \pi(B) \sum_{j=\ell+1}^h g_1(j)\frac{\pi(V_j\cap B)}{\pi(B)}.\]
	Applying Lemma \ref{lemmaconcave} with $f(j)=\sqrt{g_1(j+\ell)}$ and $w(j)=\frac{\pi(V_{j+\ell}\cap B)}{\pi(B)}$ for $j=0,\ldots,h-\ell$, gives
	\[\mathbb{E}_\pi(\tau_D)=\pi(B) \mathbb{E}_\omega(f^2)\geq \frac{\pi(B)}{7}f(h-\ell)^2= \frac{\pi(B)}{7} \mathbb{E}_v(\tau_{v_\ell}),\]
	where $v$ is a leaf at level $h$. Hence, inequality (\ref{eq:case1aldous}) and Lemma \ref{lemmacoupling} yield
	\begin{equation}\label{eq:aldous2}
		\trel \geq \frac{\pi(D)}{7} \mathbb{E}_v(\tau_{v_\ell})\geq \frac{1}{14}\cdot \frac{1}{4(\frac{2}{5}+2)}\tmix\geq\frac{1}{135}\tmix.
	\end{equation}
	Next, assume that $ \mathbb{E}_v(\tau_{v_\ell})\leq\frac{5}{2}\mathbb{E}_o(\tau_{v_\ell})$, where $v$ is a leaf of $T$ at level $h$. In particular, $\degree_0=1$. We will use (\ref{gamma}) again to bound the relaxation time. Define $g_2 \colon V \longrightarrow \mathbb{R}$ by
	\[ g_2(x)=\left\{
	\begin{array}{ll}
		i & \mbox{ if } x=v_i \mbox{ for } i \in \{0,\ldots,\ell\};\\
		\ell & \mbox{ otherwise.}\\
	\end{array}
	\right.
	\]
	On the one hand, we can compute the Dirichlet form of $g_2$ as follows:
	\begin{align*}
		\mathcal{E}(g_2)&=\frac{1}{2}\sum_{x,y \in V}|g_2(x)-g_2(y)|^2\pi(x)P(x,y)=\frac{1}{2} \sum_{k=0}^{\ell} \pi(v_k) \sum_{\substack{i=0 \\ i\neq k}}^\ell P(v_k,v_i)\\
		&=\frac{1}{4} \sum_{k=0}^{\ell-1} \pi(v_k) + \frac{1}{2}\frac{\pi(v_\ell)}{2\degree_\ell}=\frac{1}{4} \frac{\ell}{|E|}.
	\end{align*}
	On the other hand, the variance of $g$ can be estimated as
	\begin{align*}
		\operatorname{Var}_\pi(g_2)&= \sum_{x \in V} |g_2(x)-\mathbb{E}_\pi(g_2)|^2 \pi(x)\geq \sum_{x \in S} |g_2(x)-\mathbb{E}_\pi(g_2)|^2 \pi(x)\\
		&\geq\frac{1}{|E|}\left (\frac{1}{2}|0- \mathbb{E}_\pi(g_2)|^2 + \frac{1}{2} |\ell-\mathbb{E}_\pi(g_2)|^2 + \sum_{k=1}^{\ell-1} |k-\mathbb{E}_\pi(g_2)|^2 \right ),
	\end{align*}
	where $S$ denotes the initial segment of the tree (see Figure \ref{figure1}). Now, if we study the above expression as a function of $\mathbb{E}_\pi(g_2)$, it is easy to see that the minimum is attained when $\mathbb{E}_\pi(g_2)=\frac{\ell}{2}$, and so we have that
	\begin{align*} \operatorname{Var}_\pi(g_2)&\geq \frac{1}{|E|}\sum_{k=1}^{\ell} \left |k-\frac{\ell}{2}\right |^2
		=\frac{1}{|E|} \frac{\ell^3+2\ell}{12}\geq \frac{1}{12} \frac{\ell^3}{|E|}.
	\end{align*}
	The expected time for the lazy simple random walk to go from $o$ to $v_\ell$ is $2\ell^2$ (see Section 10.4 in \cite{levin2017markov}), so in view of (\ref{gamma}) and Lemma \ref{lemmacoupling}, we conclude that
	\begin{equation}\label{final2} \trel\geq \frac{\operatorname{Var}_\pi(g_2)}{\mathcal{E}(g_2)}\geq\frac{4}{12} \ell^2 \geq \frac{4}{12} \cdot \frac{1}{8(1+5)}\tmix= \frac{1}{144}\tmix.
	\end{equation}
	This proves the first part of the statement. Now, the second part follows from (\ref{criterion}).
\end{proof}

\section{Hitting times}\label{section4}
In this section, we present the proof of Theorem \ref{theo:hitting}, which will be broken into several lemmata. Given $x$ and $y$ vertices of $T$, first we study the case when $y$ is an ancestor of $x$. By a Markov chain on an (undirected) graph $G$ with transition matrix $P$, we mean that transition probabilities satisfy $P(x,y)>0$ if and only if $\{x,y\}$ is an edge of $G$. Given vertices $x\neq y \in G$, let $G_{x,y}$ denote the union of the connected components of $G\setminus\{x\}$ not containing $y$ (see Figure \ref{figure2} for an example when $G$ is a tree). We start with the following general lemma.

\begin{lemma}\label{lemma RS}
	Let $(X_t)$ be an irreducible Markov chain on a graph $G$ starting at $x\in G$. Given $y \in G\setminus\{x\}$, we have that $\tau_{y}=R+S$, where $R$ is the time needed for the chain restricted to $G\setminus G_{x,y}$ to go from $x$ to $y$, and $S$ is a random variable positively correlated to $R$ satisfying
	\begin{equation}\label{eq:lemmaRS} \Var(S)\geq
		\mathbb{E}(S)^2.
	\end{equation}
\end{lemma}

\begin{proof}
	If $G_{x,y}=\emptyset$, we can take $S=0$ and $R=\tau_y$. Otherwise, let $N$ be the number of times that the chain visits the set $G_{x,y}$ and comes back to $x$ before hitting $y$. Note that $N+1$ follows a geometric distribution with parameter $p=\probb{x}{\tau_y<\tau_{G_{x,y}}}$. Consider random variables $\tau_1,\ldots,\tau_N$ representing the length of those excursions. Then $S=\sum_{j=1}^N \tau_j$ gives the total time that the chain spends on $G_{x,y}$ before hitting $y$. Thus, $R=\tau_y - S$ gives the time that the chain spends on $G\setminus G_{x,y}$ before hitting $y$. It is easy to see that $\mathbb{E}(R|N)$ and $\mathbb{E}(S|N)$ are independent random variables that increase as $N$ increases, which implies that $R$ and $S$ are positively correlated. Indeed, we have
	\[ \mathbb{E}(RS)=\mathbb{E}(\mathbb{E}(RS|N))=\mathbb{E}(\mathbb{E}(R|N)\mathbb{E}(S|N))\geq \mathbb{E}(\mathbb{E}(R|N)) \mathbb{E}(\mathbb{E}(S|N)) = \mathbb{E}(R)\mathbb{E}(S),\]
	where the second equality comes from the independence of $\mathbb{E}(R|N)$ and $\mathbb{E}(S|N)$, and the inequality comes from Chebyshev's inequality for monotone random variables. Finally, the law of total variance yields
	\[ \Var(S)\geq \Var(\mathbb{E}(S|N))=\Var(N\mathbb{E}(\tau_1))=\frac{1-p}{p^2}\mathbb{E}(\tau_1)^2=(1-p)^{-1}\mathbb{E}(S)^2\geq \mathbb{E}(S)^2.\qedhere\]
\end{proof}

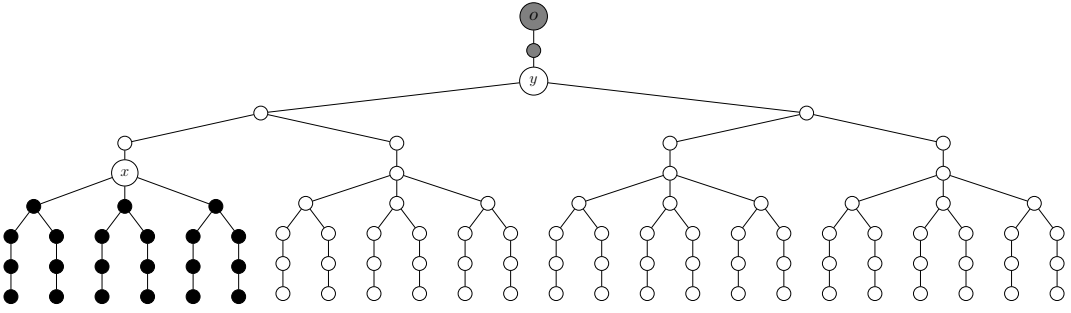
\begin{figure}[t]
	\centering
	\scalebox{0.41}{
		\begin{forest}
			for tree={math content,
				grow=south,
				circle, draw, minimum size=2ex, inner sep=4.5pt,
				s sep=10mm, l sep=3mm
			}
			[o,fill={gray},inner sep=5,font=\huge[,fill={gray}[y,inner sep=5,font=\LARGE[[[x, inner sep=5,font=\LARGE[,fill={black}[,fill={black}[,fill={black}[,fill={black}]]][,fill={black}[,fill={black}[,fill={black}]]]][,fill={black}[,fill={black}[,fill={black}[,fill={black}]]][,fill={black}[,fill={black}[,fill={black}]]]][,fill={black}[,fill={black}[,fill={black}[,fill={black}]]][,fill={black}[,fill={black}[,fill={black}]]]]]][[[[[[]]][[[]]]][[[[]]][[[]]]][[[[]]][[[]]]]]]][,fill={white}[,fill={white}[,fill={white}[,fill={white}[,fill={white}[,fill={white}[,fill={white}]]][,fill={white}[,fill={white}[,fill={white}]]]][,fill={white}[,fill={white}[,fill={white}[,fill={white}]]][,fill={white}[,fill={white}[,fill={white}]]]][,fill={white}[,fill={white}[,fill={white}[,fill={white}]]][,fill={white}[,fill={white}[,fill={white}]]]]]][,fill={white}[,fill={white}[,fill={white}[,fill={white}[,fill={white}[,fill={white}]]][,fill={white}[,fill={white}[,fill={white}]]]][,fill={white}[,fill={white}[,fill={white}[,fill={white}]]][,fill={white}[,fill={white}[,fill={white}]]]][,fill={white}[,fill={white}[,fill={white}[,fill={white}]]][,fill={white}[,fill={white}[,fill={white}]]]]]]]]]]
	\end{forest}}
	\caption{Example of a spherically symmetric tree in which $y$ is an ancestor of $x$. The set $G_{x,y}$ considered in Lemma \ref{lemma RS} corresponds to the black vertices. When studying the hitting time from $x$ to $y$, gray vertices can be discarded.}
	\label{figure2}
	
\end{figure}

Let $(X_t)$ be the lazy simple random walk on a spherically symmetric tree $T$ of height $h$. Identifying all vertices at the same level we obtain the associated birth-and-death chain $(\tilde{X}_t)$ defined on $\{0,\ldots,h\}$ (see Figure \ref{figure3}). We write $\tilde{P}$ for its transition matrix and $\tilde{\pi}$ for its stationary distribution. Given $x$, $y \in T$ so that $y$ is an ancestor of $x$, the hitting time of $y$ starting from $x$ for $(X_t)$ and for $(\tilde{X}_t)$ is the same. Moreover, we may assume that the state $y$ is absorbing since it does not change the hitting time of $y$.	The idea is to use Lemma \ref{lemma RS} to reduce the study of the hitting time $\tau_y$ to the case when $h=d(x,y)$. Then we can decompose the hitting time of $y$ as a sum of independent geometric variables. The continuous-time version of this decomposition was proved by Karlin and McGregor (see \cite[Equation (45)]{Karlin1959coincidence}), and reproved by Keilson in \cite{Keilson1979markov}. Here we use its discrete-time version, which was given by Fill (see \cite[Theorem 1.2]{Fill2009the}).

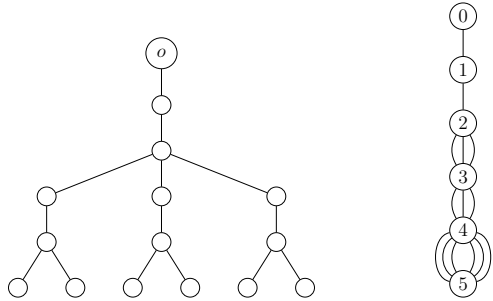
\begin{figure}[b]
	\centering
	
	\scalebox{0.5}{\subfloat{{\begin{forest}
					for tree={math content,
						grow=south,
						circle, draw, minimum size=2ex, inner sep=5pt,
						s sep=10mm, l sep=7mm
					}
					[o,font=\LARGE[[[[[][]]][[[][]]][[[][]]]]]]
	\end{forest}}}}
	\qquad \qquad
	\scalebox{0.5}{\subfloat{{\begin{forest}
					for tree={math content,
						grow=south,
						circle, draw, minimum size=2ex, inner sep=3pt,
						s sep=10mm, l sep=7mm,
						if level=5{no edge,}{}
					}
					[0,name=n0,inner sep=3,font=\LARGE[1,name=n1,font=\LARGE[2,name=n2,font=\LARGE[3,name=n3,font=\LARGE[4,name=n4,font=\LARGE[5,name=n5,font=\LARGE]]]]]]
					\draw (n2) to [bend right=30] (n3);
					\draw (n2) to [bend left=30] (n3);
					\draw (n3) to [bend right=30] (n4);
					\draw (n3) to [bend left=30] (n4);
					\draw (n4) to [bend right=30] (n5);
					\draw (n4) to [bend left=30] (n5);
					\draw (n4) to [bend right=55] (n5);
					\draw (n4) to [bend left=55] (n5);
					\draw (n4) to [bend right=80] (n5);
					\draw (n4) to [bend left=80] (n5);
	\end{forest} }}}
	\caption{Example of a spherically symmetric tree and its associated birth-and-death chain, seen as a segment with multiple edges.}
	\label{figure3}
\end{figure}

\begin{lemma}\label{lemmabirth}
	Let $(\tilde{X}_t)$ be a lazy birth-and-death chain defined on $\{0,\ldots,h\}$ satisfying $p_j\geq q_j$ for every $j \in \{0,\ldots,h-1\}$, where  $p_j$ and $q_j$ denote the birth and death probabilities, respectively. Then, 
	\[\Var_n(\tau_0)\geq \frac{1}{121}\mathbb{E}_n(\tau_0)^2 \quad \forall\, n \in  \{0,\ldots,h\}.\]
\end{lemma}

\begin{proof}
	First, take $n=h$. As we just observed, we may assume that $0$ is absorbing. Let $P_n$ be the corresponding sub-stochastic matrix which is the restriction of $\tilde{P}$ to $\{1,\ldots,n\}$, and let $\gamma_1\geq\ldots\geq\gamma_{n}$ be its eigenvalues. Using Theorem 1.2 in \cite{Fill2009the} to write $\tau_0$ as a sum of $n$ geometric random variables, whose parameters are $1-\gamma_j$ for $j=1,\ldots,n$, gives
	\begin{equation}\label{eq:Fill} \Var_{n}(\tau_0) = \sum_{i=1}^{n} \frac{\gamma_i}{(1-\gamma_i)^2} \geq \frac{\gamma_1}{(1-\gamma_1)^2}.
	\end{equation}
	Define $f\colon \{0,\ldots,n\}\longrightarrow \mathbb{R}$ by $f(t)=\mathbb{E}_t(\tau_0)$ for every $t \in\{0,\ldots,n\}$. From (\ref{eq:dirichlet}) we see that the Dirichlet form of $f$ can be bounded above by $\mathbb{E}_n(\tau_0)$. Moreover, since $p_j\geq q_j$ for every $j \in \{0,\ldots,h-1\}$, then the sequence $(\tilde{\pi}(j))_{j=1}^{h-1}$ is increasing and $\tilde{\pi}(0)\leq \tilde{\pi}(h)$. Thus, we can apply Lemma \ref{lemmaconcave} to the function $f$ with $w(j)=\tilde{\pi}(j)$ for $j \in \{0,\ldots,n\}$ to obtain that $\mathbb{E}_{\tilde{\pi}}(f^2)$ can be bounded below by $\frac{1}{7}\mathbb{E}_{n}(\tau_0)^2$. Thus, the Rayleigh-Ritz formula (cf., e.g., \cite[\S 90]{Halmos1974finite}) yields
	\begin{equation}\label{eq:trel}
		\frac{1}{1-\gamma_1} \geq \frac{\mathbb{E}_{\tilde{\pi}}(f^2)}{\mathcal{E}(f)}\geq \frac{1}{7}\mathbb{E}_{n}(\tau_0).
	\end{equation}
	Finally, since the chain is lazy, by using Perron-Frobenius theorem we deduce that $\gamma_1\geq \frac{1}{2}$, so (\ref{eq:Fill}) and (\ref{eq:trel}) gives
	\begin{equation}\label{eq:R}
		\Var_{n}(\tau_0)\geq \frac{1}{98} \mathbb{E}_{n}(\tau_0)^2.
	\end{equation}
	Assume now that $n<h$. Using Lemma \ref{lemma RS} we can write $\tau_0=R+S$, where $R$ is the time needed for the chain restricted to $\{0,\ldots,n\}$ starting at $n$ to hit $0$. On the one hand, if $\mathbb{E}(R)\geq 10 \mathbb{E}(S)$ we have
	\[\Var_{n}(\tau_0) \geq \Var(R)\geq \frac{1}{98}\mathbb{E}(R)^2\geq\frac{10^2}{98(11^2)}\mathbb{E}_{n}(\tau_0)^2\geq \frac{1}{119} \mathbb{E}_{n}(\tau_0)^2,\]
	where the first inequality follows from the positive correlation of $R$ and $S$, and the second inequality follows from (\ref{eq:R}). On the other hand, if $\mathbb{E}(R)\leq 10\mathbb{E}(S)$, we have
	\[\Var_{n}(\tau_0) \geq \Var(S)\geq \mathbb{E}(S)^2\geq\frac{1}{11^2}\mathbb{E}_{n}(\tau_0)^2=\frac{1}{121}\mathbb{E}_{n}(\tau_0)^2,\]
	where the first inequality uses the positive correlation of $R$ and $S$, and the second inequality uses (\ref{eq:lemmaRS}).
\end{proof}

The next result follows immediately from Lemma \ref{lemmabirth}.

\begin{corollary}\label{yancestor}
	Let $T$ be a spherically symmetric tree. Let $x$, $y \in T$ such that $y$ is an ancestor of $x$. Then, the lazy simple random walk on $T$ satisfies
	\[ \Var_x(\tau_y) \geq \frac{1}{121} \mathbb{E}_x(\tau_y)^2.\]
\end{corollary}

Finally, the case when $x$ is an ancestor of $y$ follows from the next lemma, which shows that hitting times for random walks on graphs are not concentrated when the starting point $x$ is a central node. Indeed, it proves something more general. 

\begin{lemma}\label{lem:central}
	Let $(X_t)$ be an irreducible Markov chain on a graph $G$. Then, 
	\[ \Var_x(\tau_y) \geq \pi(G_{x,y})^2 \mathbb{E}_x(\tau_y)^2 \quad \forall \, x \neq y \in G.\]
\end{lemma}

\begin{proof}
	Take $x\neq y \in G$. Using Lemma \ref{lemma RS} we write $\tau_y=R+S$, where $S$ is the time that the chain spends in $G_{x,y}$ until it hits $y$ (see Figure \ref{figure4}). We can apply \cite[Lemma 10.5]{levin2017markov} with $\mu=\nu=\delta_x$ and $\tau=\tau_{x,y}$ to get
	\[ \mathbb{E}(S)\geq\sum_{z \in G_{x,y}} \mathbb{E}_x\left (\sum_{t=0}^{\tau_{x,y}}1_{\{X_t=z\}}\right )=\sum_{z\in G_{x,y}} t_{x\leftrightarrow y}\pi(z)=\pi(G_{x,y})t_{x\leftrightarrow y}\geq \pi(G_{x,y})\mathbb{E}_x(\tau_y).\]
	The result now follows from (\ref{eq:lemmaRS}).
\end{proof}

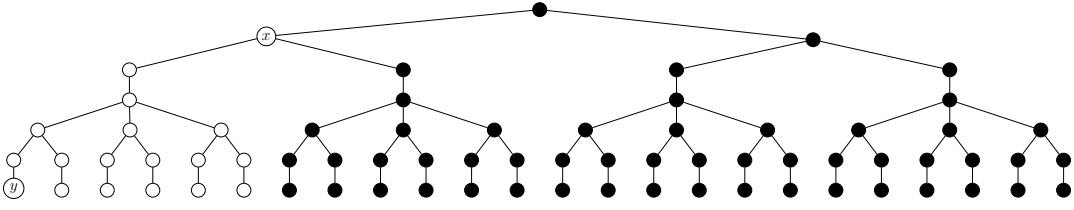
\begin{figure}[t]
	\centering
	\scalebox{0.41}{
		\begin{forest}
			for tree={math content,
				grow=south,
				circle, draw, minimum size=2ex, inner sep=4.5pt,
				s sep=10mm, l sep=3mm
			}
			[,fill={black}[x,inner sep=2.5,font=\LARGE[[[,fill={white}[,fill={white}[y,fill={white},inner sep=2.5,font=\LARGE]][,fill={white}[,fill={white}]]][,fill={white}[,fill={white}[,fill={white}]][,fill={white}[,fill={white}]]][,fill={white}[,fill={white}[,fill={white}]][,fill={white}[,fill={white}]]]]][,fill={black}[,fill={black}[,fill={black}[,fill={black}[,fill={black}]][,fill={black}[,fill={black}]]][,fill={black}[,fill={black}[,fill={black}]][,fill={black}[,fill={black}]]][,fill={black}[,fill={black}[,fill={black}]][,fill={black}[,fill={black}]]]]]][,fill={black}[,fill={black}[,fill={black}[,fill={black}[,fill={black}[,fill={black}]][,fill={black}[,fill={black}]]][,fill={black}[,fill={black}[,fill={black}]][,fill={black}[,fill={black}]]][,fill={black}[,fill={black}[,fill={black}]][,fill={black}[,fill={black}]]]]][,fill={black}[,fill={black}[,fill={black}[,fill={black}[,fill={black}]][,fill={black}[,fill={black}]]][,fill={black}[,fill={black}[,fill={black}]][,fill={black}[,fill={black}]]][,fill={black}[,fill={black}[,fill={black}]][,fill={black}[,fill={black}]]]]]]]
	\end{forest}}
	\caption{Example of a spherically symmetric tree in which $x$ is an ancestor of $y$. The set $G_{x,y}$ considered in Lemma \ref{lem:central} corresponds to the black vertices.}
	\label{figure4}
\end{figure}

\begin{proof}[Proof of Theorem \ref{theo:hitting}]
	Assume first that $x\neq y$, so we have $\tau_y=\tau_y^+$ if the chain starts at $x$. Let $q \in T$ be the nearest common ancestor of $x$ and $y$. First, the time to go from $x$ to $y$ is the sum of the time to go from $x$ to $q$ and the time to go from $q$ to $y$, which are independent random variables, and so $\Var_x(\tau_y)=\Var_x(\tau_q)+\Var_q(\tau_y)$. Next, Corollary \ref{yancestor} gives $\Var_x(\tau_q)\geq \frac{1}{121}\mathbb{E}_x(\tau_q)^2$. If $y=q$ we have finished. For $y\neq q$, Lemma \ref{lem:central} applied to $q$ and $y$ gives $\Var_q(\tau_y)\geq\pi(T_{q,y})^2\mathbb{E}_q(\tau_y)^2$. Observe that if $q$ is either the closest branching point to the root or a descendant of it, then $\pi(T_{q,y})\geq \frac{1}{2}$. Consequently,
	\[\Var_x(\tau_y)\geq \frac{1}{121}(\mathbb{E}_x(\tau_q)^2 +\mathbb{E}_q(\tau_y)^2)\geq \frac{1}{242}\left (\mathbb{E}_x(\tau_q)+\mathbb{E}_q(\tau_y)\right )^2=\frac{1}{242} \mathbb{E}_x(\tau_y)^2.\]
	Otherwise, we must have $q=x$. Let $(x=u_0,\ldots,u_{d(x,y)}=y)$ be the (unique) path joining $x$ and $y$. Let $u_\ell$ be the closest branching point to $x$. If there are no branching points, set $u_\ell=y$. As before, we have that $\Var_x(\tau_y)=\Var_x(\tau_{u_\ell})+ \Var_{u_\ell}(\tau_y)$. Applying Corollary \ref{yancestor} to $x$ and $u_\ell$ gives $\Var_x(\tau_{u_\ell})\geq \frac{1}{121}\mathbb{E}_x(\tau_{u_\ell})^2$. If $u_\ell=y$ there is nothing else to prove. Otherwise, Lemma \ref{lem:central} applied to $u_\ell$ and $y$ gives $\Var_{u_\ell}(\tau_y)\geq \pi(T_{u_\ell,y})^2\mathbb{E}_{u_\ell}(\tau_y)^2$, where $\pi(T_{u_\ell,y})\geq \frac{1}{2}$. Therefore,
	\[\Var_x(\tau_y)\geq \frac{1}{121}(\mathbb{E}_x(\tau_{u_\ell})^2 + \mathbb{E}_{u_\ell}(\tau_y)^2) \geq\frac{1}{242}\left (\mathbb{E}_x(\tau_{u_\ell})+\mathbb{E}_{u_\ell}(\tau_y)\right )^2=
	\frac{1}{242}\mathbb{E}_x(\tau_y)^2.\]
	Asumme now that $x=y$. Observe that $\mathbb{E}_x(\tau_x^+)-1=\sum_{z\in V} P(x,z)\mathbb{E}_z(\tau_x)$. Thus,
	\[\Var_x(\tau_x^+)\geq \sum_{z \in V} P(x,z)\Var_z(\tau_x)\geq\sum_{z \in V} P(x,z)\frac{\mathbb{E}_z(\tau_x)^2}{242}\geq \frac{(\mathbb{E}_x(\tau_x^+)-1)^2}{121}\geq\frac{\mathbb{E}_x(\tau_x^+)^2}{484}.\qedhere \]
\end{proof}

\section{Stability under rough isometries}\label{sectionrough}
Consider two graphs $G=(V,E)$ and $G'=(V',E')$ with graph distances $d$ and $d'$. A function $\phi\colon V \longrightarrow V'$ is a \textit{rough isometry} if there are $\alpha>0$ and $\beta>0$ such that,
\begin{equation}\label{rough}\alpha^{-1} d(x,y)- \beta\leq d'(\phi(x),\phi(y))\leq \alpha d(x,y) + \beta \quad \forall \, x , y \in V,
\end{equation}
and such that every vertex of $G'$ is within distance $\beta$ of the image of $V$. If such a function exists, we say that $G$ and $G'$ are roughly isometric.

The relaxation time of the lazy simple random walk on a graph is preserved, up to a constant, under rough isometries. Moreover, this constant only depends on the degree of the graph and the constants $\alpha$ and $\beta$ given by the rough isometry. This fact follows from the path comparison method described in \cite[Theorem 13.20]{levin2017markov} and Lemma 3.14 in \cite{diaconis1996logarithmic}. Although the mixing time is not stable under rough isometries in general (see \cite{ding2013sensitivity}), it was proved in \cite{peres2015mixing} that for (weighted) trees, the mixing time of the lazy simple random walk is stable under bounded perturbation of the edge weights. More generally, Theorem 1.1 in \cite{addario-berry2018mixing} shows that for general trees, the mixing time is preserved up to a constant under rough isometries. This constant only depends on the degree of the graph and the constants $\alpha$ and $\beta$. Recall that condition (\ref{criterion}) is equivalent to cutoff for lazy simple random walks on trees. Thus, the previous observations give the following result.

\begin{proposition}\label{prop}
	Let $(T_n)_{n\in\mathbb{N}}$ and $(T'_n)_{n\in\mathbb{N}}$ be sequences of trees with bounded degree $\Delta$. Assume that $T_n$ and $T'_n$ are roughly isometric with $\alpha$ and $\beta$ not depending on $n$. Then, the lazy simple random walk on $(T_n)_{n\in\mathbb{N}}$ exhibits cutoff if and only if it does on $(T'_n)_{n\in\mathbb{N}}$.
\end{proposition}




\end{document}